\documentclass[12pt]{amsart}
\usepackage{graphicx}
\newtheorem{thm}{Theorem}[section]
\newtheorem{assumption}[thm]{Assumption}
\newtheorem{cor}[thm]{Corollary}
\newtheorem{lem}[thm]{Lemma}
\newtheorem{prop}[thm]{Proposition}
\theoremstyle{definition}
\newtheorem{defn}[thm]{Definition}
\theoremstyle{rem}
\newtheorem{rem}[thm]{Remark}
\numberwithin{equation}{section}

\def\high{1}
\def\low{0}

\def\p0{p_0}
\def\p1*{p^{*}}
\def\pp1*{p'}

\def\t1eps{t^{\epsilon}_1}

\def\fracp0{\frac{1-p_0}{p_0}}
\def\fracp00{\frac{p_0}{1-p_0}}
\def\fracp1*{\frac{\p1*}{1-\p1*}}

\def\vg{\bar{\nu}_{\high}}
\def\vvg{\bar{\nu}_{\low}}

\def\vi{\nu_{\high}}
\def\vvi{\nu_{\low}}

\def\vh{\nu_{\high} (dh)}
\def\vvh{\nu_{\low} (dh)}

\def\Ph{P^h}

\def\F{\mathcal{F}}

\def\KK{\textbf{K}}
\def\KR{K_R}
\def\KS{K_S}
\def\K{K}
\def\key{k}

\renewenvironment{proof}[1][Proof]{\textbf{#1.} }{\ \rule{0.5em}{0.5em}}

\begin{document}




\title{Bandit problems with L\'{e}vy processes}

\author{Asaf Cohen}\thanks{The School of Mathematical Sciences, Tel-Aviv University, Tel-Aviv 69978, Israel,
cohenasa@post.tau.ac.il, web: https://sites.google.com/site/asafcohentau/} 
\author{Eilon Solan}\thanks{The School of Mathematical Sciences, Tel-Aviv University, Tel-Aviv 69978, Israel, eilons@post.tau.ac.il, web: www.math.tau.ac.il/$\sim$eilons}

\begin{abstract}
Bandit problems model the trade-off between exploration and
exploitation in various decision problems. We study two-armed bandit
problems in continuous time, where the risky arm can have two types:
High or Low; both types yield stochastic payoffs generated by a
L\'{e}vy process. We show that the optimal strategy is a cut-off
strategy and we provide an explicit expression for the cut-off and
for the optimal payoff.
\end{abstract}



\maketitle

%

\section{Introduction}\label{s:introduction}

A variation of this model was studied in filtering theory by Kalman
and Bucy (1961) \cite{Kalman1961} and Zakai (1969) \cite{Zakai1969}.
They analyze a more general model where a decision maker observes a
function of a diffusion process with an additional noise, which is
formulated as a Brownian motion. They provide equations that the
posterior or the unnormalized posterior distribution at time $t$
satisfies. Bandit problems, first described in Robbins (1952)
\cite{Robbins}, are a mathematical model for studying the trade
between exploration and exploitation. In its simplest formulation, a
decision maker (DM) faces $N$ slot machines (called \emph{arms}) and
has to choose one of them at each time instance. Each slot machine
delivers a reward when and only when chosen. The reward's
distribution of each slot machine is drawn according to an unknown
distribution, which itself is drawn according to a known probability
distribution from a set of known distributions. The DM's goal is to
maximize his total discounted payoff. The trade-off that the DM
faces at each stage is between exploiting the information that he
already has, that is, choosing the arm that looks optimal according
to his information, and exploring the arms, that is, choosing a
suboptimal arm to improve his information about its payoff
distribution. A good strategy for the DM will involve phases of
exploration and phases of exploitation. In exploration phases the DM
samples the rewards of the various machines and learns their
rewards' distributions. In exploitation phases the DM samples the
machine whose reward's distribution so far is best until evidence
shows that its reward's distribution is not as good as expected.

Bandit problems have been applied to various areas, like economics,
control, statistics, and learning; see, e.g., Chernoff (1972)
\cite{Chernoff}, Rothschild (1974) \cite{Rothschild}, Weitzman
(1979) \cite{Weitzman}, Roberts and Weitzman (1981) \cite{Roberts},
Lai and Robbins (1984) \cite{Lai}, Bolton and Harris (1999)
\cite{Bolton}, Moscarini and Squintani (2010) \cite{Moscarini},
Keller, Rady, and Cripps (2005) \cite{Keller2005}, Bergemann and
V\"{a}lim\"{a}ki (2006) \cite{Bergemann}, Besanko and Wu (2008)
\cite{Besanko}, and Klein and Rady (2011) \cite{Klein}.

Gittins and Jones (1979) \cite{Gittins} proved that in discrete time
the optimal strategy of the DM has a particularly simple form: at
every period the DM calculates for each arm an index, which is a
real number, based on past rewards of that arm, and chooses the arm
with the highest index. It turns out that to calculate the index of
an arm it is sufficient to consider an auxiliary problem with two
arms: an arm for which the index is calculated and an arm that
yields a constant payoff. The former arm is termed the \emph{risky}
arm, because its payoff distribution is not known, while the latter
is termed the \emph{safe} arm. The literature therefore focuses on
such problems, called two-armed bandit problems.\footnote{In the
literature these problems are also called one-armed bandit
problems.}

Once the optimality of the index strategy is guaranteed, one looks
for the relation between the data of the problems and the index.
Explicit formulas for the index when the payoff is one of two
distributions that have a simple form have been established in the
literature. Berry and Friestedt (1985) \cite{Berry} provide the
solution to the problem in several cases, e.g., in discrete time
when the payoff distribution is one of two Bernoulli distributions,
and in continuous time when the payoff distribution is one of two
Brownian motions. By studying the dynamic programming equation that
describes the problem in continuous time, Keller, Rady, and Cripps
(2005) \cite{Keller2005} and Keller and Rady (2010)
\cite{Keller2010} provided an explicit form for the index when time
is continuous and the payoff's distribution
is Poisson.%
\footnote{These authors also studied the strategic setup in which
several DMs have the same set of arms and their arms' payoff
distributions are the same (and unknown), and they compared the
cooperative solution to the non-cooperative solution.}

In the present paper we study two-armed bandit problems in
continuous time and provide an explicit solution when the payoff
distribution of the risky arm is one of two L\'{e}vy processes. We
assume that one distribution, called {\em High}, dominates the
other, called {\em Low}, in a strong sense (see Assumption
\ref{assumption} below). To eliminate trivial cases, we assume that
the expected payoff generated by the safe arm is lower than the
expected payoff generated by the High distribution, and higher than
the expected payoff generated by the Low distribution.

In discrete time, under these assumptions the optimal strategy is a
cut-off strategy: the DM keeps experimenting as long as the
posterior belief that the distribution is High is higher than some
cut-off point, and, once the posterior probability that the
distribution is High falls below the cut-off point, the DM switches
to the safe arm forever. We extend this result to our setup, and
prove that when the two payoff distributions are L\'{e}vy processes
that satisfy several requirements, the optimal strategy is a cut-off
strategy. Moreover, we provide an explicit expression for the
cut-off point in terms of the data of the problem. When
particularized to the models studied by Bolton and Harris (1999)
\cite{Bolton}, Keller, Rady, and Cripps (2005) \cite{Keller2005},
and Keller and Rady (2010) \cite{Keller2010} our expression reduces
to the expressions that they obtained.

Apart from unifying previous results, our characterization shows
that the special form of the optimal payoff derived by Bolton and
Harris (1999) \cite{Bolton} and Keller, Rady, and Cripps (2005)
\cite{Keller2005} is valid in a general setup: the optimal payoff is
the sum of the expected payoff, if no information is available, and
an option value that measures the expected gain from the ability to
experiment. It also shows that the data of the problem can be
divided into information-relevant parameters and payoff-relevant
parameters; the information-relevant parameters can be summarized in
a single real number, and the payoff-relevant parameters are the
expectations of the processes that contribute to the DM's payoff.
Finally, the characterization allows one to derive comparative
statics on the optimal cut-off and payoff. For example, as the
discount rate increases, or the signals become less informative, the
cut-off point increases and the DM's optimal payoff decreases.

The rest of the paper is organized as follows. In Section \ref{s:the
model} we present the model, the types of strategies that we allow,
and the assumptions that the payoff process should satisfy. In
Section \ref{s:posterior} we define the process of posterior belief
and we develop its infinitesimal generator. In Section \ref{s:value}
we present the value function, and in Section \ref{s:HJB} we present
the Hamilton--Jacobi--Bellman (HJB) equation. The main result, which
characterizes the optimal strategy and the optimal payoff of the DM,
is formulated and proved in Section \ref{s:main}. The Appendix
contains the proofs of several results that are needed in the paper.

\section{The model}\label{s:the model}
\subsection{Reminder about L\'{e}vy processes}\label{s:reminder}
L\'{e}vy processes are the continuous-time analog of discrete-time
random walks with i.i.d.~increments. A \emph{L\'{e}vy process}
$X=(X(t))_{t\geq 0}$ is a continuous-time stochastic process that
(a) starts at the origin: $X(0)=0$, (b) admits c\`{a}dl\`{a}g
modification,\footnote{That is, it is continuous from the right, and
has limits from the left: for every $t_0$, the limit $X(t_0
-):=\underset{t\nearrow t_0}{\lim}X(t)$ exists a.s. and $X(t_0)
=\underset{t\searrow t_0}{\lim}X(t)$. } and (c) has stationary
independent increments. Examples of L\'{e}vy processes are a
Brownian motion, a Poisson process, and a compound Poisson process.

Let $(X(t))$ be a L\'{e}vy process. For every Borel measurable set
$A\subseteq \mathbb{R}\backslash \{0\}$, and every $t\geq 0$, let
the Poisson random measure $N(t ,A)$ be the number of jumps of
$(X(t))$ in the time interval $[0,t]$ with jump size in $A$:
\[
N(t,A)=\sharp\{0\leq s \leq t \mid \Delta X(t): = X(s) - X(s-) \in
A\}.
\]
By Applebaum (2004) \cite{Applebaum}, one can define a Borel measure
$\nu$ on $\mathcal{B}(\mathbb{R}\backslash\{0\})$ by
\[\nu (A) := E[N(1,A)] = \int N(1,A)(\omega)dP(\omega),\] where $(\Omega,P)$ is the underlying probability space.
The measure $\nu (A)$ is called the \emph{L\'{e}vy measure} of
$(X(t))$, or the \emph{intensity measure} associated with $(X(t))$.

We now present the L\'{e}vy--It\={o} decomposition of L\'{e}vy
processes. Let $(X(t))$ be a L\'{e}vy process; then there exists a
constant $b\in \mathbb{R}$, a Brownian motion $\sigma Z(t)$ with
standard deviation $\sigma$, and an independent Poisson random
measure $N_{\nu}(t,dh)$ with the associated L\'{e}vy measure $\nu$
such that, for each $t\geq 0 $,
\begin{equation}\notag
X(t) = b t +\sigma Z(t) + \underset{h>|1|}{\int} hN_{\nu}(t,dh) +
\underset{h \leq |1|}{\int} h\widetilde{N}_{\nu}(t,dh),
\end{equation}
where $\widetilde{N}_{\nu} (t,A) := N_{\nu}(t,A)-t\nu (A)$ is the
\emph{compensated Poisson random measure}. This representation is
called the \emph{L\'{e}vy--It\={o} decomposition} of the L\'{e}vy
process $(X(t))$. Thus, a L\'{e}vy processes is characterized by the
triplet $\langle b,\sigma, \nu \rangle$.

If the L\'{e}vy process has finite expectation for each $t$, that
is, $E|X(t)|<\infty$ for all $t\geq 0 $, then the L\'{e}vy process
can be represented as
\begin{equation}\notag
X(t) = \mu t +\sigma Z(t) + \underset{\mathbb{R}\backslash \{0\}
}{\int} h\widetilde{N}_{\nu}(t,dh) ;
\end{equation}
that is, $X(t)$ can be represented as the sum of a linear drift, a
Brownian motion, and an independent purely discontinuous
martingale\footnote{A \textit{purely discontinuous process} is a
process that is orthogonal to all continuous local martingales. For
details, see Jacod and Shiryaev (1987, Ch.~I, Definition 4.11)
\cite{Jacod}.} (see Sato (1999, Theorem 25.3) \cite{Sato1987}).
\begin{rem} Even though the process $(X(t))$ has finite expectation, it is possible that
\[
E\left[\underset{\mathbb{R}\backslash \{0\}}{\int}
|h|N_{\nu}(t,dh)\right]= \infty,
\]
which means that the expectation of the sum of the jumps of $X(t)$
in any time interval is infinite.
\end{rem}

\subsection{L\'{e}vy bandits}\label{s:finite}
A DM operates a two-armed bandit machine in continuous time, with a
safe arm that yields a constant payoff $\varrho$, and a risky arm
that yields a stochastic payoff $(X(t))$ that depends on its type
$\theta$. The risky arm can be of two types, High or Low. With
probability $p_0=p$ the arm's type is High, and with probability
$1-p$ it is Low. If the type is High (resp. Low) we set $\theta = 1$
(resp. $0$). The process $(X(t))$ is a L\'{e}vy process with the
triplet $\langle\mu_\theta,\sigma,\nu_\theta\rangle$; that is, the
L\'{e}vy--It\={o} decomposition of $(X (t))$ is $X(t) = \mu_\theta t
+ \sigma_\theta Z(t) + \underset{\mathbb{R}\backslash \{0\} }{\int}
h\widetilde{N}_{\nu_\theta}(t,dh) $. Formally, for
$\theta\in\{0,1\}$, let $(X_\theta (t))$ be a L\'{e}vy process with
triplet $\langle\mu_\theta,\sigma,\nu_\theta\rangle$ and let
$\theta$ be an independent Bernoulli random variable with parameter
$p$. The process $(X(t))$ is defined to be $(X_\theta(t))$.
We denote by $P_p$ the probability measure over the space of
realized paths that corresponds to this description.
From now on, unless mentioned otherwise, all the expectations are
taken under the probability measure $P_p$.
%

\subsection{Strategies}\label{s:strategies}
We adopt the concept of continuous-time strategies first introduced
by Mandelbaum, Shepp, and Vanderbei (1990) \cite{Mandelbaum}. An
\emph{allocation strategy} $\KK= \{ \KK (t)\mid t\in[0,\infty)\}$ is
a nonnegative stochastic process $\KK(t)=(\KR(t),\KS(t))$ that
satisfies
\begin{align}\tag{K1}\label{T1}
&\KR(0) = \KS(0) = 0, \ \textrm{and}\ (\KR(t)) \;\textrm{and}\
(\KS(t))\ \textrm{are nondecreasing processes},\\\tag{K2}\label{T2}
&\KR(t)+\KS(t)=t, \;\; t\in[0,\infty), and \\\tag{K3}\label{T3}
&\{\KR(t)\leq s\} \in \F_s^X, \; \; t,s\in[0,\infty),\notag
\end{align}
where $\F_s^X$ is the sigma-algebra generated by the process
$(X(t))_{t\leq s}$. The interpretation of an allocation process is
that the quantity $\KR(t)$ (resp.~$\KS(t)$) is the time that the DM
devotes to the risky arm (resp.~safe arm) during the time interval
$[0,t)$. The process $(\KK(t))$ is basically a two-parameter time
change of the two-dimensional process $(X(t),\varrho t)$.

Below we will define a stochastic integral with respect to
$(X_\theta(t))$, and therefore we assume throughout that both
L\'{e}vy processes $ (X_{1} (t))$ and $(X_{0} (t))$ have finite
quadratic variation, that is, $E[X_\theta^2(t)]<\infty $ for every
$t\geq 0$ and each $\theta\in\{0,1\}$. It follows that the processes
$(X_\theta(t))$ have finite expectation.
\begin{assumption}$\\$\label{finite_quadratic}
A1. $E[X^2_\theta (1)]=\mu_\theta^2 +\sigma^2 +{\int}
h^2\nu_\theta(dh)<\infty.$
\end{assumption}

For every $(t,p)\in [0,\infty)\times[0,1]$, every real-valued
function $S:\mathbb{R}\rightarrow \mathbb{R}$, and every pair of
Markov processes $(H_1(t))$ and $(H_2(t))$ with respect to the
filtration $(\F_t^X)_{t\geq 0}$ under both $P_0$ and $P_1$ such that
$E \left[S(\int_t^{\infty} H_1(s) dH_2(s))\mid \theta \right]$ are
well defined for both $\theta\in\{0,1\}$, we define the following
expectation operator:
\begin{align}
\notag E^{t,p} \left[S\left(\int_t^{\infty} H_1(s)
dH_2(s)\right)\right]:&= \left.p E \left[S\left(\int_t^{\infty}
H_1(s) dH_2(s)\right)\right|  \theta =
1\right]\\
\label{Etp} &+\left.(1-p) E \left[S\left(\int_t^{\infty} H_1(s)
dH_2(s)\right)\right| \theta = 0\right].\\\notag
\end{align}

Using this notation, the expected discounted payoff from time $t$
onwards under allocation strategy $\KK$ when the prior belief at
time $t$ is $p_t =p$ can be expressed as
\begin{align}\label{VKtp}
V_{\KK}(t,p) &:= E^{t,p} \left[\int_t^{\infty} r
e^{-rs}dY(\KK(s))\right],
\\\notag
\end{align}
where $Y(\KK(s)):=X(\KR(s))+\varrho \KS(s)$. The goal of the DM is
to maximize $V_{\KK}(0,p)$. Let
\begin{align}\label{defineU}
U(t,p) := \underset{\KK}{\sup}V_{\KK}(t,p)\\\notag
\end{align}
be the maximal payoff the DM can achieve from time $t$ onwards,
given that the prior belief at time $t$ is $p_t =p$. As we show in
Theorem \ref{Up1} below, under proper assumptions the DM has an
optimal strategy, so in fact the supremum in Eq. (\ref{defineU}) is
achieved. Moreover, we give explicit expressions for both the
optimal strategy and the optimal value function $U(t,p)$.

\begin{rem}\label{key}
By Conditions (\ref{T1}) and (\ref{T2}), $\KR$ and $\KS$ are
Lipschitz and thus absolutely continuous. Therefore, there exists a
two-dimensional stochastic process $\KK'(t)=\dfrac{d\KK}{dt}(t) =
(\KR'(t), \KS'(t))$ such that $\KK(t)=\int_0^t \KK'(s) ds$. To
simplify notation we denote $\K(t):=\KR (t)$, and $\key(t)
:=\K_R'(t)$. Hence, $\KK(t)=(\K(t),t-\K(t))$ and
$\KK'(t)=(\key(t),1-\key(t))$. The process $(\key(t))$ may be
interpreted as follows: At each time instance $t$, the DM chooses
$\key(t)$ (resp. $ 1 - \key(t)$), the proportion of time in the
interval $[t,t+dt)$ that is devoted to the risky arm (resp. the safe
arm). The process $(\key(t))$ will be treated as a stochastic
control parameter of the process $(X(t))$. Denote by $\F_{\K(t)}$
the sigma-algebra generated by $(X(\K(s)))_{s\leq t}$.
\end{rem}

\begin{defn}
An \emph{admissible control strategy} $(\key(t,\omega))$ is any
predictable process such that $0\leq \key\leq 1$ with probability
$1$, and such that the process $\K(t)=\int_0^t \key(s) ds$
satisfies\footnote{Since $0\leq \key\leq 1$, it follows that $\K(t)$
satisfies Conditions (\ref{T1}) and (\ref{T2}) as well.} Condition
(\ref{T3}). Denote by $\Upsilon$ the set of all admissible control
strategies.
\end{defn}

In the sequel we will not distinguish between the allocation
strategy $(\K(t))$ and the corresponding admissible control strategy
$(\key(t))$.

\subsection{Assumptions}
If the DM could deduce the type of the risky arm by observing the
payoff of the risky arm in an infinitesimal time interval, then an
almost-optimal strategy is to start at time $0$ with the risky arm,
and switch at time $\delta$ to the safe arm if the type of the risky
arm is Low, where $\delta>0$ is a small real number. Throughout the
paper we make the following assumption, which implies that the DM
cannot distinguish between the two types in any infinitesimal time.

\begin{assumption}\label{assumption}$\\$
A2. $\sigma_{\high}=\sigma_{\low}$.\\
A3. $|\nu_1(\mathbb{R}\setminus \{0\}) - \nu_0(\mathbb{R}\setminus \{0\})| < \infty$.\\
A4. $|\int h (\nu_1 (dh) - \nu_0 (dh))|<\infty$.\\
\end{assumption}

Assumption A2 states that the Brownian motion component of both the
High type and the Low type have the same standard deviation. By
Revuz and Yor (1999, Ch.~I, Theorem 2.7) \cite{Revuz} the realized
path reveals the standard deviation and therefore if Assumption A2
does not hold then the DM can distinguish between the arms in any
infinitesimal time interval. Assumption A3 states that the
difference between the L\'{e}vy measures is finite and Assumption A4
states that the difference between the expectation of the jump part
of the processes is finite. Otherwise, by comparing the jump part of
the processes, the DM could distinguish between the arms in any
infinitesimal time interval.

We also need the following assumption, which states that the High
type is better then the Low type in a strong sense.

\begin{assumption}\label{assumption}$\\$
A5. $\mu_0<\varrho<\mu_1$.\\
A6. For every $A\in\mathcal{B}(\mathbb{R}\setminus \{0\}),\; \; \vvi
(A)\leq\vi (A)$.
\end{assumption}

Assumption A5 merely says that the High (resp. Low) type provides
higher (resp. lower) expected payoff than the safe arm. Assumption
A6 is less innocuous; it requires that the L\'{e}vy measure of the
High type dominates the L\'{e}vy measure of the Low type in a strong
sense. Roughly, jumps of any size $h$, both positive and negative,
occur more often (or at the same rate) under the High type than
under the Low type. A consequence of this assumption is that jumps
always provide good news, and (weakly) increase the posterior
probability of the High type.

\begin{rem}\label{moments}
Although we require that the zeroth and first moments of $(\nu_1
(dh) - \nu_0 (dh))$ are finite (Assumptions A3 and A4), this
requirement is not made for moments of higher order, since, by
Assumption \ref{finite_quadratic}, $\int_{\mathbb{R}\setminus \{0\}}
h^2\nu_\theta(dh)\leq \mu_\theta^2+\sigma^2 +
\int_{\mathbb{R}\setminus \{0\}}h^2\nu_\theta(dh) =
E[X^2_\theta(1)]<\infty$, for $\theta\in \{0,1\}$.
\end{rem}
\section{The posterior belief}
\label{s:posterior}
\subsection{Motivation}
At time $t=0$ the type $\theta$ is chosen randomly with $P(\theta=1)
= 1 - P(\theta=0) = p$. The DM does not observe $\theta$, but he
knows the prior $p$ and observes the controlled process
$(X(\K(t)))$. Let $p_t:=P(\theta=1\mid \F_{\K(t)})$ be the posterior
belief at time $t$ that the risky arm's type is High under the
allocation strategy $(\KK(t))$. The following proposition asserts
that the payoff $V_\KK (t,p)$ can be expressed solely by
the\footnote{$ p_{t-}$ is the posterior belief at time $K(t)-$.}
posterior process $(p_{t-})$ and the allocation strategy $(\KK(t))$.
This representation motivates the investigation of the posterior
process.

\begin{prop}\label{replacep}
For every allocation strategy $\KK$,
\begin{align}\notag
V_\KK(t,p) &= E^{t,p} \left[\int_t^{\infty} r e^{-rs} [(\mu_1 p_{s-}
+\mu_0 (1-p_{s-}) )\key(s) + \varrho(1-\key(s))] ds\right].
\\\notag
\end{align}
\end{prop}

\begin{proof}
We will prove the following series of equations, which proves the
claim:
\begin{align}\label{11}
E^{t,p}&\left[\int_t^\infty re^{-rs}dY(\KK(s))\right]= E^{t,p}\left[
\underset{x\rightarrow \infty}{\lim}\int_t^x
re^{-rs}dY(\KK(s))\right] \\\label{12} &= \underset{x\rightarrow
\infty}{\lim}E^{t,p}\left[\int_t^x re^{-rs}dY(\KK(s))\right]
\\\label{13} &= \underset{x\rightarrow \infty}{\lim}E^{t,p}\left[\int_t^x
re^{-rs}[(\mu_1 p_{s-} +\mu_0 (1-p_{s-}) )\key(s) +
\varrho(1-\key(s))] ds\right]\\\label{14} &=
E^{t,p}\left[\underset{x\rightarrow \infty}{\lim}\int_t^x
re^{-rs}[(\mu_1 p_{s-} +\mu_0 (1-p_{s-}) )\key(s) +
\varrho(1-\key(s))] ds\right]\\\label{15} &=
E^{t,p}\left[\int_t^\infty re^{-rs}[(\mu_1 p_{s-} +\mu_0 (1-p_{s-})
)\key(s) + \varrho(1-\key(s))] ds\right].\\\notag
\end{align}
Eqs.~(\ref{11}) and (\ref{15}) hold by the definition of the
improper integral. Let $\textit{[}X(\K(s))\textit{]}$ be the
quadratic variation of the time-changed process $(X(\K(s)))$. From
the It\={o} isometry and Kobayashi (2011, pages 797--799)
\cite{Kobayashi}, it follows that
\begin{align}\notag
&E^{t,p} \left[\int_t^\infty re^{-rs}dX(\K(s)) - \int_t^x
re^{-rs}dX(\K(s))   \right]^2\\\notag &= E^{t,p} \left[\int_x^\infty
re^{-rs}dX(\K(s)) \right]^2  = E^{t,p} \left[\int_x^\infty
(re^{-rs})^2 d\textit{[}X(\K(s))\textit{]} \right]  \\\notag & =
E\left.\left[\int_x^\infty (re^{-rs})^2
d\textsl{[}X(\K(s))\textsl{]} \right|\theta = 1 \right] p + E
\left.\left[\int_x^\infty (re^{-rs})^2 d\textit{[}X(\K(s))\textit{]}
\right|\theta = 0 \right] (1-p) \\\notag &=E
\left.\left[\int_x^\infty (re^{-rs})^2 c_1 d\K(s) \right|\theta = 1
\right] p + E \left.\left[\int_x^\infty (re^{-rs})^2 c_0 d\K(s)
\right|\theta = 0 \right] (1-p), \\\notag
\end{align}
where $c_\theta =E[X^2_\theta(1)]= \mu^2_\theta +\sigma^2 +
\int_{\mathbb{R}\setminus \{0\}}h^2\nu_\theta(dh)$, for
$\theta\in\{0,1\}$. Hence, $\int_t^x re^{-rs}dX(\K(s))$ convergence
to $\int_t^\infty re^{-rs}dX(\K(s))$ in $L^2$ and Eq.~(\ref{12})
follows. Eq.~(\ref{13}) follows from Corollary \ref{martingales}
(part C1) in the appendix. Eq.~(\ref{14}) follows from the dominated
convergence theorem, since for every $x\geq t$,
\begin{align}\notag
\left| \int_t^x re^{-rs}[(\mu_1 p_{s-} +\mu_0 (1-p_{s-}) )\key(s) +
\varrho(1-\key(s))] ds\right|\leq \max\{|\mu_0|, |\mu_1|\} .
\end{align}
\end{proof}

\subsection{Formal definition of the posterior belief}\label{s:formal_definition}
An elegant formulation of the Bayesian belief updating process was
presented by Shiryaev (1978, Ch.~4.2) \cite{Shiryaev} to a model in
which the observed process is a Brownian motion with unknown drift
and extended later to a model in which the observed process is a
Poisson process with unknown rate in Peskir and Shiryaev (2000)
\cite{Peskir2000}.\footnote{Similar work has been done in the
disorder problem; see, e.g., Shiryaev (1978) \cite{Shiryaev}, Peskir
and Shiryaev (2002) \cite{Peskir2002}, and Gapeev (2005)
\cite{Gapeev}. } We follow this formulation and extend it to the
time-changed L\'{e}vy process. For every $p\in[0,1]$, the
probability measure $P_p$ satisfies $ P_p  = p P_1+(1-p)P_0.$ An
important auxiliary process is the Radon--Nikodym density, given by
\begin{align}\notag
\varphi_t := \frac{d(P_0\mid\F_{\K(t)})}{d(P_1\mid \F_{\K(t)})},
\;\;t\in[0,\infty).
\end{align}
\begin{lem}
For every $t\in[0,\infty)$,
\begin{align}\notag
p_t = \frac{p}{p +(1-p)\varphi_t}.
\\\notag
\end{align}
\end{lem}
\begin{proof}
Define the following Radon--Nikodym density process
\begin{align}\notag
\pi_t = p\frac{d(P_1\mid\F_{\K(t)})}{d(P_p\mid\F_{\K(t)})},
\;\;t\in[0,\infty),
\end{align}
where $P_p(\cdot\mid\F_{\K(t)}) = p P_1(\cdot\mid\F_{\K(t)})+(1 - p)
P_0(\cdot\mid\F_{\K(t)})$. From the definition of $(\varphi_t)$ it
follows that $\pi_t = \frac{p}{p +(1-p)\varphi_t}$. Therefore, it is
left to prove that $p_t=\pi_t$ for every $t\in[0,\infty)$. Let
$A\in\F_{\K(s)}$ where $s\geq t$. The following series of equations
yields that $p_t=\pi_t$ for every $t\in[0,\infty)$:
\begin{align}\label{31}
E^p[\chi_A p_s|\F_{K(t)}] &= E^p[\chi_A
E^p[\chi_{\{\theta=1\}}|\F_{K(s)}]|\F_{K(t)}] \\\label{32}& =
E^p[\chi_{A\cap\{\theta=1\}} |\F_{K(t)}] \\\label{33}& = p
E^1[\chi_{A} |\F_{K(t)}] \\\label{34}& = E^p[\chi_{A}\pi_s
|\F_{K(t)}],
\end{align}
where $\chi_A=1$ if $A$ is satisfied and zero otherwise.
Eq.~(\ref{31}) follows from the definition of $p_t$. Eq.~(\ref{32})
follows since $s\geq t$, and, therefore,
$\F_{K(s)}\supseteq\F_{K(t)}$. Eq.~(\ref{33}) follows from the
definition of the probability measure $P_p$, and Eq.~(\ref{34})
follows from the property of the Radon--Nikodym density $\pi_t$.
\end{proof}


By Jacod and Shiryaev (1987, Ch.~III, Theorems 3.24 and 5.19)
\cite{Jacod}, the process $(\varphi_t)$ admits the following
representation:
\begin{align}\notag
\varphi_t =\exp \left\{ \beta\sigma Z(\K(t))+ (\vg-\vvg
-\frac{1}{2}\beta^2 \sigma^2) \K(t) + \underset{\mathbb{R}\setminus
\{0\}}{\int}\ln \left(\frac{\nu_0}{\nu_1}(h)\right)N(\K(t),dh)
\right\},\notag
\end{align}
where $\beta := \frac{\mu_0-\mu_1-\underset{\mathbb{R}\setminus
\{0\}}{\int}h(\nu_0-\nu_1)(dh)}{\sigma^2}$ and $\vg-\vvg :=
\underset{\mathbb{R}\setminus \{0\}}{\int}(\nu_1(dh)-\nu_0(dh))$. By
Assumption A6, $\vg-\vvg$ is finite and the Radon--Nikodym
derivative $\frac{\nu_0}{\nu_1}(h)$ exists.\footnote{To ensure the
existence of the Radon--Nikodym derivative one does not need the
full power of Assumption A6. Its full power will be used for the
proof of Theorem \ref{Up1}.}
\begin{rem}\label{Markovian}
1. Let $B_\infty\in \mathcal{B}(\mathbb{R}\setminus \{0\})$ be a
maximal set (up to $\nu_1 $-measure zero) such that
$\nu_1(B_\infty)\geq 0=\nu_0 (B_\infty)$. Occurrence of a jump from
$B_\infty$ indicates that the risky arm is High.
By definition, $\varphi_t=0$ after such a jump and therefore $p_t=1$.\\
2. By ignoring jumps from $B_\infty$, $(\ln(\varphi_t))$ is a
L\'{e}vy process\footnote{However, it is not a L\'{e}vy process
under $P_p$ for $0<p<1$, since it is not time-homogeneous.} with
time change $(\K(t))$, under both $P_0$ and $P_1$ with respect to
the filtration generated by $(\varphi_t)$, which coincides with
$(\F_{\K(t)})$. From the one-to-one correspondence between
$\varphi_t$ and $p_t$ it follows that $p_t$ is a Markov process.
Therefore, our optimal control problem falls in the scope of optimal
control of Markov processes. Hence, we can limit the allocation
strategies to \emph{Markovian allocation strategies}, or
equivalently to \emph{Markovian control strategies}, which we define
as follows.
\end{rem}
\begin{defn}
A control strategy $(\key(t,\omega))$ is \emph{Markovian} if it
depends solely on the Markovian process $(t,p_{t-})$. That is,
$\key(t,\omega) = \key(t,p_{t-})$. Denote by $\Upsilon_M$ the set of
all Markovian control strategies.
\end{defn}
\begin{rem}
A convenient way to understand the ``Girsanov style'' process
$(\varphi_t)$ is to examine the process $(X(t))$. We may assume that
\begin{align}\notag
&X(t)=\mu_1 t+\sigma Z(t) + \underset{\mathbb{R}\setminus
\{0\}}{\int}h \widetilde{N}_{\nu_1}(t,dh),
\end{align}
where, under $P_1$, $(Z(t))$ is a Brownian motion and the last term
is a purely discontinuous martingale. By the definition of $\beta$,
the same process can be represented as
\begin{align}\notag
&X(t)=\mu_0 t +\sigma (\beta\sigma t +Z(t)) +
\underset{\mathbb{R}\setminus \{0\}}{\int} h
\widetilde{N}_{\nu_0}(t,dh).
\end{align}
Under $P_0$ the process $(\beta\sigma t +Z(t))$ is a Brownian motion
and  the last term is a purely discontinuous martingale. For
details, see Jacod and Shiryaev (1987, Ch.~III) \cite{Jacod}.
\end{rem}

\subsection{The infinitesimal operator}
An important tool in the proofs is the \emph{infinitesimal operator}
of the process $(t,p_t)$ with respect to the Markovian control
strategy $\key$, which we will calculate in this section. The
infinitesimal operator (or infinitesimal generator) of a stochastic
process is the stochastic analog of a partial derivative (see
{\O}ksendal, 2000). In this section we calculate the infinitesimal
operator of the process $(t,p_t)$ with respect to the Markovian
control strategy $\key$, which we will use in the proof of Theorem
\ref{Up1}. By It\={o}'s formula (see, e.g., Kobayashi (2011), pages
797--799) \cite{Kobayashi}, the posterior process $(p_t)$ solves the
following stochastic differential equation:
\begin{align}\label{dP1}
dp_t = & ~[\beta^2\sigma^2(1-p_{t-})^2
p_{t-}-(\vg-\vvg)p_{t-}(1-p_{t-})]d\K(t) \\\notag
       & -p_{t-}(1-p_{t-})\beta\sigma dZ(\K(t)) \\\notag
       & +p_{t-}(1-p_{t-})\underset{h\in{\mathbb{R}\setminus \{0\}}}{\int}
\frac{1-\frac{\nu_0}{\nu_1}(h)}{p_{t-}+\frac{\nu_0}{\nu_1}(h)(1-p_{t-})}
N(d\K(t),dh)  \\\notag
      = & ~ p_{t-}(1-p_{t-})\left[ -\beta dM(\K(t))-(\vg-\vvg)d\K(t) \right.\\\notag
      & + \left.\underset{h\in{\mathbb{R}\setminus \{0\}}}{\int} \frac{1-\frac{\nu_0}{\nu_1}(h)}{p_{t-}+\frac{\nu_0}{\nu_1}(h)(1-p_{t-})}
         N(d\K(t),dh) \right]  ,\notag
\end{align}
where
%
%
\begin{align}\notag
M(\K(t)) &= X(\K(t)) -\underset{\mathbb{R}\setminus \{0\}}{\int} h
\widetilde{N}_{\nu_0}(\K(t),dh) - \mu_0\K(t)+\beta\sigma^2 \int_0^t
p_{s-} dK(s)\notag
\end{align}
is a martingale under $P_p$ with respect to $\F_{\K(t)}$; see
Corollary \ref{martingales} (part C2) in the appendix.

The first term on the right-hand side of ~(\ref{dP1}),
$-p_{t-}(1-p_{t-})\beta dM(\K(t))$, is the contribution of the
continuous part of the payoff process to the change in the belief,
while the second term, $- p_{t-}(1-p_{t-})(\vg-\vvg)d\K(t)$, is the
contribution of the fact that no jump occurred. This latter
contribution is negative due to Assumption A6. If a jump of size $h$
occurs during the interval $[t,t+dt)$, then the contribution of the
jump is $\Ph_t - p_{t-}$, where, $\Ph_t:=
\frac{p_{t-}\nu_1(dh)}{p_{t-}\nu_1(dh) + (1-p_{t-} ) \nu_0(dh)}$ is
the Bayesian update of the probability that the risky arm is High
given that a jump of size $h$ occurs. By Assumption A6, for every
$0< p<1 $ we have $P_p(  p_t < \Ph_t)=1$.

To calculate the infinitesimal operator of the process $(t,p_t)$
with respect to the Markovian control strategy $\key$ we apply
It\={o}'s formula\footnote{$C^{1,2}$ is the set of all functions
$f:[0,\infty)\times[0,1]\rightarrow \mathbb{R}$, which are $C^1$ in
their first coordinate, and $C^2$ in their second coordinate.} for
$f(t,p)\in C^{1,2}([0,\infty)\times[0,1])$ and obtain
\begin{align}\label{ito_for_f}
f(t,p_t)=& f(0,p_0)  + \int_0^t {f_t(s,p_{s-})ds}
                        + \int_0^t {f_p (s,p_{s-})dp_s}\\\notag
                        &+ \frac{1}{2}\int_0^t {f_{pp}(s,p_{s-}) p_{s-}^2(1-p_{s-})^2\beta^2\sigma^2d\K(s)}\\\notag
                        &+\sum_{s\leq t} \left[  f(s,p_s)- f(s,p_{s-})-f_p(s,p_{s-}) \Delta(p_s)  \right]\\\notag
                        =& f(0,p_0)  + \int_0^t {f_t(s,p_{s-})ds}\\\notag
                        &- \int_0^t {f_p (s,p_{s-})[ (\vg-\vvg)p_{s-}(1-p_{s-}) ]d\K(s)}\\\notag
                        &+ \frac{1}{2}\int_0^t {f_{pp}(s,p_{s-}) p_{s-}^2(1-p_{s-})^2\beta^2\sigma^2d\K(s)}\\\notag
                        &+ \int_{s=0}^t \int_{h\in{\mathbb{R}\setminus \{0\}}}  \left(f\left(s,\frac{p_{s-}}{p_{s-}+(1-p_{s-})\frac{\nu_0}{\nu_1}(h)}\right) -
                        f(s,p_{s-})\right)\\\notag
                        \cdot &(p_{s-}\nu_1 (dh) + (1-p_{s-})\nu_0 (dh))d\K(s)  \\\notag
                        &- \int_0^t {f_{p}(s,p_{s-}) p_{s-}(1-p_{s-})\beta dM(\K(s))}\\\notag
                        &+ \underset{{s=0}}{\int^t} \int_{h\in{\mathbb{R}\setminus \{0\}}} \left(f\left(s,\frac{p_{s-}}{p_{s-}+(1-p_{s-})\frac{\nu_0}{\nu_1}(h)}\right)  - f(s,p_{s-})\right) \\\notag
                        \cdot & [N(d\K(s),dh) - (p_{s-}\nu_1 (dh) + (1-p_{s-})\nu_0 (dh))d\K(s)  ].\\\notag
\end{align}

The fifth and sixth terms on the right-hand side of
Eq.~(\ref{ito_for_f}) are stochastic integrals with respect to
martingales and therefore they are local martingales (see Jacod and
Shiryaev (1987, Ch.~I, Theorem 4.40) \cite{Jacod}). The seventh term
is a stochastic integral with respect to a compensated random
measure, as will be shown in Corollary \ref{martingales} (parts C2
and C3) in the appendix. Therefore, it is a local martingale (see
Jacod and Shiryaev (1987, Ch.~II, Theorem 1.8) \cite{Jacod}). Hence,
by taking expectations of both sides it follows that the
infinitesimal operator of the process $(t,p_t)$ with respect to the
Markovian control strategy $(\key(t,p))$ is
\begin{align}\notag
(\mathbb{L}^\key f)(t,p) =&  f_t(t,p)
                        -(\vg-\vvg)p(1-p)f_p (t,p)\key(t,p)
                        +\frac{1}{2} \beta^2\sigma^2 f_{pp}(t,p) p^2(1-p)^2\key(t,p)\\\notag
                        &+\int_{\mathbb{R}\setminus \{0\}} { \left(f\left(t,\frac{p}{p+(1-p)\frac{\nu_0}{\nu_1}(h)}\right) - f(t,p)\right)
                        (p \nu_1 (dh) + (1-p)\nu_0 (dh))\key(t,p)  }.\\\notag
\end{align}
When $f$ is a function of $p$ only, we will use the same notation
for the infinitesimal operator of the process $(p_t)$ with respect
to the time-homogeneous Markovian control strategy $\key(p)$.
Specifically,
\begin{align}\label{Lkp}
(\mathbb{L}^\key f)(p) =&  -(\vg-\vvg)p(1-p)f' (p)\key(p)
                        +\frac{1}{2} \beta^2\sigma^2 f''(p) p^2(1-p)^2\key(p)\\\notag
                        &+\int_{\mathbb{R}\setminus \{0\}} { \left(f\left(\frac{p}{p+(1-p)\frac{\nu_0}{\nu_1}(h)}\right) - f(p)\right)
                        (p \nu_1 (dh) + (1-p)\nu_0 (dh))\key(p)  }.\\\notag
\end{align}

\section{The value function}
\label{s:value} In the next section we will introduce the
Hamilton--Jacobi--Bellman (HJB) for our problem. The value function
$U(t,p)$ is not $C^2$ in its second coordinate, and therefore we
need to formalize the optimal problem differently. Additionally to
the Markovian control strategy $\key(t,p)$, we will add an
artificial stopping time $\tau$ to the new strategy space. This new
form will help us solve the HJB although $U(t,p)$ is not $C^2$. We
start with a few basic properties of the value function $U(t,p)$.
\begin{prop}\label{convex}
For every fixed $t\geq 0$, the function $p\mapsto U(t,p)$ is
monotone, nondecreasing, convex, and continuous.

\end{prop}
\begin{proof}
Fix for a moment an allocation strategy $\KK$. By Definition
\ref{Etp} and Eq.~(\ref{VKtp}) the expected discounted payoff from
time $t$ onwards under strategy $\KK$ when $p_t =p$ is
\begin{align}\notag
V_\KK(t,p) &= E^{t,p} \left[\int_t^{\infty} r e^{-rs}
dY(\KK(s))\right]\\\notag
              &= \left.p E\left[\int_t^{\infty} r e^{-rs} dY(\KK(s))\right|  \theta =
1\right]+\left.(1-p) E \left[\int_t^{\infty} r e^{-rs}
dY(\KK(s))\right| \theta = 0\right].\\\notag
\end{align}
For every fixed $t\geq 0$ the function $p\mapsto V_{\KK}(t,p)$ is
linear. Therefore $U(t,p)$, as the supremum of linear functions, is
convex. By always choosing the safe arm, the DM can achieve at least
$e^{-rt}\varrho$, and by always choosing the risky arm the DM can
achieve at least $ e^{-rt} (p \mu_1 + (1-p) \mu_0) $. Since
$U(t,0)=e^{-rt}\varrho$ and $U(t,1) = e^{-rt}\mu_1$, the convexity
of $U(t,p)$ implies that the function $p \mapsto U(t,p)$ is
continuous and nondecreasing in $p$.
\end{proof}

It follows from Proposition \ref{convex} that for every fixed $t\geq
0$ there is a time-dependent cut-off $p^*_t$ in $[0,1]$ such that
$U(t,p)=\varrho$ if $p\leq p^*_t$ and $U(t,p)>\varrho$ otherwise. It
follows that for every fixed $t$ the strategy $\key (t,p)\equiv 0$
that always chooses the safe arm is optimal for prior beliefs in
$[0,p^*_t]$. By this conclusion, Proposition \ref{replacep}, and
Remark \ref{key} we deduce that the optimal problem (\ref{defineU})
can be reduced to a combined optimal stopping and stochastic control
problem as follows:

\begin{align}\label{UwithW}
U(t,p) = \underset{  t\leq \tau  ,\;  \key\in\Upsilon_M  }{\sup}
E^{t,p} \left[\int_t^\tau r e^{-rs} W(p_{s-},\key(s,p_{s-}))ds +
\varrho e^{-r\tau}\right],
\end{align}
where $W(p,l):=(\mu_1 p +\mu_0 (1-p))l + \varrho(1-l)$ is the
instantaneous payoff given the posterior $p$, using the Markovian
control $l$. This representation of the value function will help us
solve the HJB equation. Denote the continuation region to be
\begin{align}\notag
D:=\{(t,p) \mid U(t,p)> \varrho e^{-rt}\}.
\end{align}
This is the region where the optimal action of the DM is to continue
(that is, $k(t,p)>0$, and $\tau>t$). The next lemma shows that the
region $D$ is invariant with respect to $t$. This means that the
optimal stopping time $\tau$ (whenever it exists) does not
depend\footnote{In fact, we will show in Theorem \ref{Up1} that an
optimal stopping time and an optimal control strategy do exist and
the optimal control $\key$ is also time-homogeneous; that is, $\key$
does not depend on $t$, and therefore the allocation strategy $\KK$
does not depend on $t$ either. } on $t$.

\begin{lem}
For every $t\geq 0$ and every $p\in[0,1]$ one has
$U(t,p)=e^{-rt}U(0,p)$. In particular, $(t,p)\in D$ if and only if
$(s,p)\in D$, for every $t,s\geq 0$ and every $p\in[0,1]$.
\end{lem}

\begin{proof}
The first claim follows from the following list of equalities:
\begin{align}\label{Utp_U0p}
U(t,p) &= \underset{  t\leq \tau  ,\; 0 \leq \key \leq 1 }{\sup}
E^{t,p} \left[\int_t^\tau r e^{-rs} W(p_{s-},\key(s,p_{s-}))ds +
\varrho e^{-r\tau}\right]
\\\notag
       &= \underset{  0\leq {\tilde{\tau}}  ,\; 0 \leq \key \leq 1 }{\sup}
E \left[\int_0^{\tilde{\tau}} r e^{-r(t+u)}
W(p_{u-}^p,\key(t+u,p_{u-}^p))du + \varrho
e^{-r(t+{\tilde{\tau}})}\right]
\\\notag
        &= e^{-rt}\underset{  0\leq \tilde{\tau}  ,\; 0 \leq \key \leq 1 }{\sup}
        E^{0,p} \left[\int_0^{\tilde{\tau}} r e^{-ru}
W(p_{u-},\key(t+u,p_{u-}))du + \varrho e^{-r{\tilde{\tau}}}\right]
\\\notag
        &=e^{-rt}U(0,p),\\\notag
\end{align}
where the second equality follows from the Markovian property of
$p_t$ (see Remark \ref{Markovian}).
\end{proof}

This lemma yields that the cut-off $p^*_t$ discussed earlier is
independent of $t$. We therefore denote it by $p^*$.

\section{The HJB equation}
\label{s:HJB} The following proposition introduces the HJB equation
for our problem.
\begin{prop}
Let $F\in C^{1}[0,1] $ be a function that satisfies
\begin{align}
F(p)\geq \varrho \;\text{for every}\;p\in[0,1].\label{Fp_varrho}
\end{align}
Define the continuation region of $F$ by
\begin{align}
C:=\{ p\in[0,1] \mid F(p)> \varrho \}.\label{continuation_region}
\end{align}
Suppose that
\begin{align}\label{C_D}
&[0,\infty)\times C=D.\\\label{C2} &F\in C^2([0,1] \backslash
\partial C) \;\; \text{with locally bounded derivatives near} \;
\partial C.\\\label{HJB_leq} &\mathbb{L}^\key F(p)+rW(p,\key(p))-rF(p)\leq
0 \; \text{on} \; [0,1]\backslash\partial C \\\notag &\text{for all
} k\in \Upsilon_M,\; \text{and all}\; p\in[0,1].\\\label{HJB_equal}
&\text{There is a control $\key^*$ for which the inequality in
Eq.~(\ref{HJB_leq})}\\\notag &\text{ holds with equality.}\\\notag
\end{align}
Then, $\key^*$ is the optimal control, $\tau_D:=\inf \{t\geq 0 \mid
F(p_t)\not\in C\}$ is the optimal stopping time, and
$U(t,p)=e^{-rt}F(p)$.
\end{prop}
Conditions (\ref{HJB_leq}) and (\ref{HJB_equal}) represent the HJB
equation in our model.
\begin{rem}
The function $F$ need not be $C^2$ at the boundary of $C$. This is
due to the representation of $U(t,p)$ in Eq.~(\ref{UwithW}) as a
combined optimal stopping and stochastic control problem. This issue
will be further discussed in the proof.
\end{rem}

\begin{proof}
Define $J(t,p):=e^{-rt}F(p)$. Then for every $(t,p)\in
[0,\infty)\times([0,1]\setminus\partial C)$,
\begin{align}\label{LJLF}
\mathbb{L}^kJ(s,p)=-re^{-rt}F(p) + e^{-rt}\mathbb{L}^kF(p).\\\notag
\end{align}

By Eq.~(\ref{C2}), $J(t,p)\in
C^{1,2}([0,\infty)\times([0,1]\setminus\partial C))$ and
$J_{pp}(s,p)$ is bounded near $[0,\infty)\times\partial C$.
Therefore, there exists a sequence $\{J^n\}_{n\geq 1}\subseteq
C^{1,2}(D)$ such that
\begin{align}\notag
J^n\rightarrow J,\;\;J^n_t\rightarrow J_t,\;\;J^n_p\rightarrow
J_p,\;\;J^n_{pp}\rightarrow J_{pp}
\end{align}
uniformly on every compact subset of
$[0,\infty)\times([0,1]\setminus\partial C)$ as $n$ goes to infinity
(see {\O}ksendal (2000, Theorem C.1) \cite{Oksendal}). Denote by
$L(t)$ the sum of the last three terms on the right-hand side of
Eq.~(\ref{ito_for_f}). The process $(L(t))$, as the sum of local
martingales, is a local martingale. Let $(\delta_m)$ be a sequence
of increasing (a.s.) stopping times that diverge (a.s.), such that
$L(\delta_m \wedge t)$ is\footnote{$a\wedge b:= \min{\{a,b\}}$.} a
martingale for every $m$. Let $\tau$ be an arbitrary stopping time
and define $\tau_m:=\tau\wedge m \wedge\delta_m$. We will prove the
following series of equations:
\begin{align}\notag
E&^{0,p}\left[ e^{-r\tau_m}F(p_{\tau_m}) \right]-F(p)\\\label{21} &=
E^{0,p}\left[ J(\tau_m,p_{\tau_m}) \right]- J(0,p)
\\\label{22}
&=
\underset{n\rightarrow\infty}{\lim}\left(E^{0,p}\left[J^n(\tau_m,p_{\tau_m})
\right]- J^n(0,p)\right)\\ \label{23}&=
\underset{n\rightarrow\infty}{\lim}E^{0,p}\left[\int_0^{\tau_m}\mathbb{L}^kJ^n(s,p_{s-})ds
\right]\\\label{24} &=
\underset{n\rightarrow\infty}{\lim}E^{0,p}\left[\int_0^{\tau_m}\mathbb{L}^kJ^n(s,p_{s-})\chi_{\{
p_{s-}\not\in\partial C\}}ds \right] \\ \label{25} &=
E^{0,p}\left[\int_0^{\tau_m}\mathbb{L}^kJ(s,p_{s-})\chi_{\{
p_{s-}\not\in\partial C\}}ds \right]\\ \label{26} &= E^{0,p}\left[
\int_0^{\tau_m}{e^{-rs}\left(-rF(p_{s-})+
\mathbb{L}^kF(p_{s-})\right)\chi_{\{ p_{s-}\not\in\partial C\}}ds}
\right]\\ \label{27} &{\leq} - E^{0,p}\left[
\int_0^{\tau_m}e^{-rs}W(p_{s-},\key(p_{s-})) \chi_{\{
p_{s-}\not\in\partial C\}}ds \right]\\ \label{28} &= - E^{0,p}\left[
\int_0^{\tau_m}e^{-rs}W(p_{s-},\key(p_{s-})) ds  \right],\\\notag
\end{align}
where, $\chi_A=1$ if $A$ is satisfied and zero otherwise.
Eq.~(\ref{21}) follows from the definition of $J(t,p)$.
Eqs.~(\ref{22}) and (\ref{25}) follow from the choice of the
sequence $J^n$. Eq.~(\ref{23}) follows from the definition of the
infinitesimal operator. By condition (\ref{C_D}), the boundary
$\partial C$ of C is a single point (specifically, it is the cut-off
point $\p1*$). Therefore, Eqs.~(\ref{24}) and (\ref{28}) follow from
Lemma \ref{spends} in the appendix.
Eq.~(\ref{26}) follows from Eq.~(\ref{LJLF}), and inequality
(\ref{27}) follows from (\ref{HJB_leq}).

By Eq. (\ref{Fp_varrho}) and the series of
Eqs.~(\ref{21})--(\ref{28}) we obtain
\begin{align}\notag
&E^{0,p}\left[ \int_0^{\tau_m}{e^{-rs}W(p_{s-},\key(p_{s-})) ds}  +
\varrho e^{-r\tau_m} \right] \\\notag &\leq E^{0,p}\left[
\int_0^{\tau_m}{e^{-rs}W(p_{s-},\key(p_{s-})) ds}  + e^{-r\tau_m}
F(p_{\tau_m}) \right]\leq F(p).
\end{align}
By taking $m\rightarrow\infty$ we deduce that
\begin{align}\label{F_geq_W}
E^{0,p}\left[ \int_0^{\tau}{e^{-rs}W(p_{s-},\key(p_{s-})) ds}  +
\varrho e^{-r\tau} \right]\leq F(p).
\end{align}
The left-hand side of Eq.~(\ref{F_geq_W}) is the payoff of the DM
using the stopping time $\tau$ and the stationary Markovian control
strategy $k$. By taking the supremum in Eq.~(\ref{F_geq_W}) it
follows that $U(0,p)\leq F(p)$ for every $0\leq p\leq 1$. To prove
the opposite inequality, apply the argument above to the stationary
Markovian control strategy $\key^*=\key^*(p_{s-})$ and the stopping
time $\tau_D$, so that the inequality in Eq.~(\ref{27}) is replaced
by an equality. By taking the limit $m\rightarrow\infty$ and by the
definition of $D$ we obtain
\begin{align}\notag
U(0,p)&\geq E^{0,p}\left[
\int_0^{\tau_D}{e^{-rs}W(p_{s-},\key^*(p_{s-})) ds}  + e^{-r\tau_D}
\varrho \right]\\\notag &= E^{0,p}\left[
\int_0^{\tau_D}{e^{-rs}W(p_{s-},\key^*(p_{s-})) ds} + e^{-r\tau_D}
F(p_{\tau_D}) \right]= F(p).
\end{align}
\end{proof}

\section{The optimal strategy}
\label{s:main} In this section we present our main result that
states that there is a unique optimal allocation strategy, and that
it is a cut-off strategy. The theorem also provides the exact
cut-off point and the corresponding expected payoff in terms of the
data of the problem. Let $\alpha^*$ be the unique solution in
$(0,\infty)$ of the equation $f(\alpha)=0$, where
\begin{equation}\label{eta}
f(\alpha) :=  \int\left(\left(\frac{\nu_0}{\nu_1}(h)\right)^{\alpha}
-1\right)\vvh + \alpha (\vg-\vvg) +\frac{1}{2}(\alpha+1)\alpha
\left(\dfrac{\mu_1-\mu_0}{\sigma}\right)^2 - r = 0.
\end{equation}
The existence and the uniqueness of such a solution are proved in
Lemma \ref{alpha} in the appendix.

\begin{thm}\label{Up1}
Denote $p^*:= \frac{\alpha^*
(\varrho-\mu_0)}{(\alpha^*+1)(\mu_1-\varrho) +
\alpha^*(\varrho-\mu_0)}$. Under Assumptions A1--A6 there is a
unique optimal strategy $\key^*$ that is time-homogeneous and is
given by
\begin{align}\label{kstar}
&\key^*=
\begin{cases}
0               &\text{if $p \leq p^*$}, \\
1               &\text{if $p > p^*$}.
\end{cases}
\end{align}
The expected payoff under $\key^*$ is
\begin{equation}\label{E:Up1}
U(0,p) = V_{\key^*} (0,p) =
\begin{cases}
\varrho                                                    &\text{if $p \leq p^*$}, \\
p\mu_1 +(1-p)\mu_0 +C_{\alpha^*} (1-p)(\frac{1-p}{p})^{\alpha^*}
&\text{if $p
> p^*$},
\end{cases}
\end{equation}
$\\$
where $C_{\alpha^*} = \frac{\varrho-\mu_0 - p^* (\mu_1 -\mu_0)}{(1-p^*) \left( \frac{1-p^*}{p^*}\right)^{\alpha^*}}.$\\
\end{thm}

We now discuss the relation between Theorem \ref{Up1} and the
results of Bolton and Harris (1999) \cite{Bolton}, Keller, Rady, and
Cripps (2005) \cite{Keller2005}, and Keller and Rady (2010)
\cite{Keller2010}. The expected payoff from the risky arm if no
information is available is
\begin{align}\notag
\int_0^\infty re^{-rs} E[X(s)]ds =\int_0^\infty re^{-rs} [p\mu_1
+(1-p)\mu_0]s ds =  p\mu_1 +(1-p)\mu_0 .
\end{align}
One can verify that the strategy $\key^* \equiv 1$ and the function
$F(p) =p\mu_1 +(1-p)\mu_0$ satisfy Condition (\ref{HJB_equal}), and
following the results of Bolton and Harris (1999) \cite{Bolton} and
Keller, Rady, and Cripps (2005) \cite{Keller2005}, one can ``guess"
that a function of the form $C(1-p)(\frac{1-p}{p})^{\alpha}$
satisfies Condition (\ref{HJB_equal}) as well. This leads to the
form of the optimal payoff that appears in Eq.~(\ref{E:Up1}). The
function $C_{\alpha^*} (1-p)(\frac{1-p}{p})^{\alpha^*}$ is the
option value for the ability to switch to the safe arm. The
parameters of the payoff processes that determine the cut-off point
$p^*$ and the optimal payoff $U(0,p)$ are the expected payoffs
$\mu_1$ and $\mu_0$. In Bolton and Harris (1999) \cite{Bolton} the
only component in the risky arm is the Brownian motion with drift.
Therefore, $\nu_i \equiv 0$, so that $ \alpha^* =
(-1+\sqrt{1+8r\sigma^2 / (\mu_1 -\mu_0)^2})/2$.
In Keller, Rady, and Cripps (2005) \cite{Keller2005}, the risky arm
is either the constant zero (Low type, so that $\vvi \equiv 0$), or
it yields a payoff $\bar{h}$ according to a Poisson process of rate
$\lambda$ (High type). If the risky arm is High, then the only
component in the L\'{e}vy--It\={o} decomposition is the purely
discontinuous component, and $\vi (\bar{h})  = \lambda$ and zero
otherwise. Therefore,  $\mu_1 = \lambda \bar{h}, \; \mu_0 = 0,$ and
$\alpha^* = r/\lambda$.
In Keller and Rady (2010) \cite{Keller2010}, the risky arm yields a
payoff $\bar{h}$ according to a Poisson process. For the High type,
the Poisson process rate is $\lambda_{\high}$, and for the Low type
the rate is $\lambda_{\low}$, where
$\lambda_{\low}<\lambda_{\high}$. The only component in the
L\'{e}vy--It\={o} decomposition is the purely discontinuous
component, and $\nu_i (\bar{h})=\lambda_i$ and zero otherwise.
Therefore,  $\mu_1 = \lambda_{\high} \bar{h}, \; \mu_0 =
\lambda_{\low} \bar{h},$ and $\alpha^*$ is the unique solution of
the equation
\begin{equation}\notag
f(\alpha) :=
\lambda_{\low}\left(\frac{\lambda_{\low}}{\lambda_{\high}}\right)^{\alpha}
+ \alpha (\lambda_{\high}-\lambda_{\low})-\lambda_{\low} - r = 0.
\end{equation}\\

\begin{proof}[\textbf{Proof of Theorem \ref{Up1}}]
Let $\p1*$, $\alpha^*$, and $C_\alpha^*$ be the parameters that were
defined in the theorem. Define the cut-off strategy $\key^*$
associated with the cut-off $\p1*$ by
\begin{equation}\key^* :=
\begin{cases}
0               &\text{if $p \leq p^*$}, \\
1               &\text{if $p > p^*$},
\end{cases}
\end{equation}
$\\$ and the function $F$ by
\begin{equation}\label{F1}
F(p) :=
\begin{cases}
\varrho                                             &\text{if $p \leq p^*$}, \\
F_1(p)                                              &\text{if $p >
p^*$},
\end{cases}
\end{equation}
$\\$ where $F_1(p):= p\mu_1 +(1-p)\mu_0 +C_{\alpha^*}
(1-p)(\frac{1-p}{p})^{\alpha^*}$. We will show that the function
$F(p)$ and the cut-off strategy $\key^*$ are the optimal payoff and
the optimal Markovian control strategy respectively. To this end, we
verify that conditions (\ref{Fp_varrho})--(\ref{HJB_equal}) are
satisfied for $F(p)$ and $\key^*$. Conditions
(\ref{Fp_varrho})--(\ref{C2}) can be easily verified for the
function $F(p)$. To verify conditions (\ref{HJB_leq}) and
(\ref{HJB_equal}) we need the full power of Assumption A6. To prove
that $F(p)$ satisfies Condition (\ref{HJB_leq}) we check separately
the cases $p\leq \p1*$ and $p>\p1*$. Fix a Markovian control
strategy $k\in\Upsilon_M$ and $p\leq\p1*$. Then,
\begin{align}\notag
\mathbb{L}^\key & F(p)+rW(p,\key(p))-rF(p)  \\\label{41} =&
-(\vg-\vvg)p(1-p)F'(p)\key(p) +\frac{1}{2} \beta^2\sigma^2 F''(p)
p^2(1-p)^2\key(p)\\\notag &+ \underset{\mathbb{R}\setminus
\{0\}}{\int} {
\left(F\left(\frac{p}{p+(1-p)\frac{\nu_0}{\nu_1}(h)}\right) -
F(p)\right) (p \nu_1 (dh) + (1-p)\nu_0 (dh))\key(p)}\\\notag &
+r[(\mu_1 p +\mu_0
(1-p))\key(p)+\varrho(1-\key(p))]-rF(p)\\\label{42}=&
\underset{G}{\int} {
\left(F_1\left(\frac{p}{p+(1-p)\frac{\nu_0}{\nu_1}(h)}\right) -
\varrho\right) (p \nu_1 (dh) + (1-p)\nu_0 (dh))\key(p)}\\\notag &
+r[(\mu_1 p +\mu_0
(1-p))\key(p)+\varrho(1-\key(p))]-r\varrho\\\label{43} \leq
&\underset{G}{\int} {
\left(F_1\left(\frac{\p1*}{\p1*+(1-\p1*)\frac{\nu_0}{\nu_1}(h)}\right)
- \varrho\right) (\p1* \nu_1 (dh) + (1-\p1*)\nu_0
(dh))\key(p)}\\\notag & +r[(\mu_1 \p1* +\mu_0
(1-\p1*))\key(p)+\varrho(1-\key(p))]-r\varrho\\\label{44} \leq &
\underset{\mathbb{R}\setminus \{0\}}{\int} {
\left(F_1\left(\frac{\p1*}{\p1*+(1-\p1*)\frac{\nu_0}{\nu_1}(h)}\right)
- \varrho\right) (\p1* \nu_1 (dh) + (1-\p1*)\nu_0
(dh))\key(p)}\\\notag & +r[(\mu_1 \p1* +\mu_0
(1-\p1*))\key(p)+\varrho(1-\key(p))]-r\varrho\\\label{45} = &\, 0,
\end{align}
where $G=\left\{ h \mid \frac{p}{p+(1-p)\frac{\nu_0}{\nu_1}(h) }>
\p1*\right\}$. Eq.~(\ref{41}) follows from Eq.~(\ref{Lkp}).
Eq.~(\ref{42}) follows since for $p<\p1*$ we have $F(p)=\varrho$,
$F'(p)=F''(p)=0$, and for $p>\p1*$ we have $F(p)=F_1(p)$. Inequality
(\ref{43}) follows from the monotonicity of $F_1$ and Assumptions A5
and A6. Inequality ~(\ref{44}) follows since $F_1(q)>\varrho$ for
every $q>\p1*$. Eq.~(\ref{45}) is satisfied for every $\key(p)$ by
the definition of $\p1*$ and $F_1$.

Fix a Markovian control strategy $k\in\Upsilon_M$ and $p > \p1*$.
Then $F(p)=F_1(p)$ and
$F\left(\frac{p}{p+(1-p)\frac{\nu_0}{\nu_1}(h)}\right)
=F_1\left(\frac{p}{p+(1-p)\frac{\nu_0}{\nu_1}(h)}\right) $, since by
Assumption A6, for every $0\leq p<1$ we have
$P_p\left(\frac{p}{p+(1-p)\frac{\nu_0}{\nu_1}(h)}
> p\right) = 1$. Therefore,
\begin{align}\notag
\mathbb{L}^\key & F(p)+rW(p,\key(p))-rF(p)  \\\label{51} =&
-(\vg-\vvg)p(1-p)F'(p)\key(p) +\frac{1}{2} \beta^2\sigma^2 F''(p)
p^2(1-p)^2\key(p)\\\notag &+ \underset{\mathbb{R}\setminus
\{0\}}{\int} {
\left(F\left(\frac{p}{p+(1-p)\frac{\nu_0}{\nu_1}(h)}\right) -
F(p)\right) (p \nu_1 (dh) + (1-p)\nu_0 (dh))\key(p)}\\\notag &
+r[(\mu_1 p +\mu_0
(1-p))\key(p)+\varrho(1-\key(p))]-rF(p)\\\label{52} =&
-(\vg-\vvg)p(1-p)F_1'(p)\key(p) +\frac{1}{2} \beta^2\sigma^2
F_1''(p) p^2(1-p)^2\key(p)\\\notag &+ \underset{\mathbb{R}\setminus
\{0\}}{\int} {
\left(F_1\left(\frac{p}{p+(1-p)\frac{\nu_0}{\nu_1}(h)}\right) -
F_1(p)\right) (p \nu_1 (dh) + (1-p)\nu_0 (dh))\key(p)}\\\notag &
+r[(\mu_1 p +\mu_0
(1-p))\key(p)+\varrho(1-\key(p))]-rF_1(p)\\\label{53} \leq & \, 0,
\end{align}
where the last inequality is satisfied for every $\key(p)$ by the
definition of $F_1$.

By the definition of the Markovian control strategy $\key^*$, for
every $p\leq \p1*$ we have $\key^*(p)=0$ and therefore
Eqs.~(\ref{41})--(\ref{45}) hold with equality, and for every
$p>\p1*$ we have $\key^*(p)=1$ and therefore
Eqs.~(\ref{51})--(\ref{53}) hold with equality by the definition of
$F(p)$. This proves condition (\ref{HJB_equal}).

Without Assumption A6 there may be a set $B$ that satisfies
$\nu_0(B)>0$, such that for every $h\in B$ and every $0\leq p<1$ one
has $P_p\left(\frac{p}{p+(1-p)\frac{\nu_0}{\nu_1}(h)} < p\right) =
1$. Thus, for every $p\in\left[p^*,
\frac{p^*}{p^*+(1-p^*)\frac{\nu_0}{\nu_1}(h)} \right)$ we need to
substitute $F_1(p)$ for $F(p)$, and $\varrho$ for
$F\left(\frac{p}{p+(1-p)\frac{\nu_0}{\nu_1}(h)}\right)$. This
problem has a higher level of complexity and it is not clear how to
approach it using the tools introduced here.
\begin{rem}
Since the process $(p_t)$ has no negative jumps, it enters the
interval $[0,\p1*]$ continuously. Therefore, we expect the value
function to be $C^1$ at the cut-off point $\p1*$. In a model where
the process $(p_t)$ has negative jumps, it can enter the interval
$[0,\p1*]$ with a jump. In this case we expect that the value
function will not be $C^1$ at the cut-off point $\p1*$. For simple
cases of L\'{e}vy processes (such as when the High (resp. Low) type
is a jump process with height $h_1$ and rate $\lambda_1$ (resp.
height $h_0$ and rate $\lambda_0$), where
$h_1\lambda_1>\varrho>h_0\lambda_0$ and $\lambda_1<\lambda_0$, so,
in particular, Assumption A3 fails) the method introduced in Peskir
and Shiryaev (2000) \cite{Peskir2000} may be useful to characterize
the optimal strategy and the value function. In the general setup, a
sample path method may be helpful to approximate the value function
via iterations (see Dayanik and Sezer (2006) \cite{Dayanik}). This
investigation is left for future research.
\end{rem}
\end{proof}

\subsection{Comparative statics}\label{Comparative Statics}
The explicit forms of the cut-off point $p^*$ and the value function
$U$ allow us to derive simple comparative statics of these
quantities. As is well known, a DM who plays optimally switches to
the safe arm later than a myopic DM, and indeed $p^*$ is smaller
than the myopic cut-off point
$p^m:=\frac{\varrho-\mu_0}{\mu_1-\mu_0}$.

Note that the cut-off point $p^*$ is an increasing function of
$\alpha^*$. As can be expected, $\alpha^*$ (and therefore also
$p^*$) increases with the discount rate $r$ and with $\vvi (dh)$,
and decreases with $\vi (dh)$ and with $\mu_1-\mu_0$. That is, the
DM switches to the safe arm earlier in the game as the discount rate
increases, as jumps provide less information, or as the difference
between the drifts of the two types increases.\footnote{Moreover,
$\alpha^*(r=0) = 0$ and $\alpha^*(r=\infty) = \infty$. }
Furthermore, as long as $p>p^*$ the value function $p \mapsto
U(0,p)$ decreases in $\alpha^*$. Thus, decreasing the discount rate,
increasing the informativeness of the
jumps, or increasing the difference between the drifts is beneficial to the DM.\\\\

\subsection{Generalization}
In our model the parameters $\mu_0$ and $\mu_1$ have two roles. By
the definition of the L\'{e}vy process $(X(t))$ they play the role
of the unknown drift. In Eq.~(\ref{UwithW}) they determine the
instantaneous expected payoff. Here we separate these two roles;
that is, we assume that the parameters that determine the
instantaneous expected payoff are not $\mu_0$ and $\mu_1$, but
rather two other parameters, $g_0$ and $g_1$ respectively. Formally,
in the definition of $W(p,l)$ in Eq.~(\ref{UwithW}) we substitute
$\mu_0$ and $\mu_1$ with other parameters, $g_0$ and $g_1$, and
observe the change in the optimal strategy and the optimal payoff.
This formulation allows us to separate the information-relevant
parameters from the payoff-relevant parameters. It also supplies an
optimal strategy and an optimal payoff in a model where the DM
receives a general linear function of the process $(X(t))$.

If we replace $W(p,l)=(\mu_1 p +\mu_0 (1-p))l+\varrho(1-l)$ with
$\widehat{W}(p,l)=(g_1 p+ g_0 (1-p))l + \varrho(1-l),$ where $g_0$
and $g_1$ are constants that satisfy $g_1>\varrho>g_0$, then the
solution of the optimization problem
\begin{align}\notag
\widehat{U}(t,p) = \underset{  t\leq \tau  ,\;  \key\in\Upsilon_M
}{\sup} E^{t,p} \left[\int_t^\tau r e^{-rs}
\widehat{W}(p_{s-},\key(s,p_{s-}))ds + \varrho e^{-r\tau}\right]
\end{align}
has a similar form to the one given in Theorem \ref{Up1}. Denote
$\widehat{p}^*:= \frac{\widehat{\alpha}^*
(\varrho-g_0)}{(\widehat{\alpha}^*+1)(g_1-\varrho) +
\widehat{\alpha}^*(\varrho-g_0)}$, where $\widehat{\alpha}^* =
\alpha^*$. Under Assumptions A1--A6, there is a unique optimal
strategy that is time-homogeneous and is given by
$$\widehat{\key}^* =
\begin{cases}
0               &\text{if $p \leq \widehat{p}^*$}, \\
1               &\text{if $p > \widehat{p}^*$}.
\end{cases}
$$
The expected payoff under $\widehat{\key}^*$ is
\begin{equation}\notag
\widehat{U}(0,p) = \widehat{V}_{\widehat{\key}^*} (0,p) =
\begin{cases}
\varrho                                                &\text{if $p \leq \widehat{p}^*$}, \\
pg_1 +(1-p)g_0 +\widehat{C}_{\alpha^*}
(1-p)(\frac{1-p}{p})^{\alpha^*} &\text{if $p
> \widehat{p}^*$},
\end{cases}
\end{equation}
$\\$ where $\widehat{C}_{\alpha^*} = \frac{\varrho-g_0 -
\widehat{p}^* (g_1 -g_0)}{(1-\widehat{p}^*)
\left( \frac{1-\widehat{p}^*}{\widehat{p}^*}\right)^{\alpha^*}}.$\\
The significance of this result is that it separates the
information-relevant parameters of the model from the
payoff-relevant parameters. The quantity $\alpha^*$ summarizes all
the information-relevant parameters, whereas $g_1$ and $g_0$ are the
only payoff-relevant parameters. For beliefs above the cut-off, the
optimal payoff is the sum of the expected payoff if the DM always
continues, $pg_1 +(1-p)g_0$, and the option value of
experimentation, which is given by $\widehat{C}_{\alpha^*}
(1-p)(\frac{1-p}{p})^{\alpha^*}$.

\section{APPENDIX}
The following lemma states that a time-changed martingale under an
allocation strategy remains a martingale.
\begin{lem}\label{lem:time_change}
Let $(M(t))$ be a martingale with respect to $\F_t$, and let
$(\KK(t))$ be an allocation strategy that satisfies (\ref{T1}),
(\ref{T2}), and (\ref{T3}). Then $(M(\K(t)))$ is a martingale with
respect to $\F_{\K(t)}$.
\end{lem}

\begin{proof}
Fix $0\leq s\leq t$. Then $\K(s)$ and $\K(t)$ are bounded stopping
times with $\K(s)\leq \K(t)$. The \textit{optional stopping theorem}
implies that $M(K(s)) = E[M(K(t))\mid\F_{\K(s)}]$, and therefore,
$(M(\K(t)))$ is indeed an $(\F_{\K(t)}, P)$-martingale.
\end{proof}

The following lemma presents the (predictable) compensator of a
process under a change of measure; see Jacod and Shiryaev (1987,
Ch.~I, Theorem 3.18) \cite{Jacod}.
\begin{lem}\label{compensator}
Under the notations of Section \ref{s:formal_definition}, let
$(H(t))$ be a stochastic process, and let $\theta$ be an independent
Bernoulli random variable with parameter $p$, such that given
$\theta$, the process $(H(t)-a_\theta t)$ is a martingale with
respect to $\F_t^H$ under $P_\theta$. Let $\tilde{p}_t := P_p
\{\theta=1|\F_t^H\}$ be the posterior belief that $\theta=1$ given
$(H(s))_{s\leq t}$ under the probability measure $P_p$. Then the
process $\left(\int_0^t (\tilde{p}_{s-}a_1 + (1-\tilde{p}_{s-})a_0)
ds\right)$ is the (predictable) compensator of the process $(H(t))$
with respect to $\F_t^H$ under the probability measure $P_p$.
\end{lem}

\begin{proof}
Plainly we have
\begin{align}\notag
&H(t)-\int_0^t (\tilde{p}_{s-}a_1 + (1-\tilde{p}_{s-})a_0) ds
\\\notag &= \theta(H(t)-a_1 t) + (1-\theta)(H(t)-a_0 t)+\int_0^t
(\theta - \tilde{p}_{s-})(a_1-a_0)ds.
\end{align}
Fix $0\leq u\leq t$. The expectation with respect to $P_p$ is
\begin{align}\notag
&E\left.\left[ H(t) - \int_0^{t} (\tilde{p}_{s-}a_1 +
(1-\tilde{p}_{s-})a_0) ds   - H(u) + \int_0^{u} (\tilde{p}_{s-}a_1 +
(1-\tilde{p}_{s-})a_0) ds   \right| \F_{u}^H\right]\\\notag &=
E\left.\left[ H(t) -H(u) - \int_u^{t} (\tilde{p}_{s-}a_1 +
(1-\tilde{p}_{s-})a_0) ds  \right| \F_{u}^H\right]\\\notag &=
E\left.\left[\theta(H(t)-a_1 t - H(u)+a_1 u )) + (1-\theta)(H(t)-a_0
t - H(u)+a_0 u)\right| \F_{u}^H\right]\\\notag &\ \ +
E\left.\left[\int_u^{t} (\theta - \tilde{p}_{s-})(a_1-a_0)ds\right|
\F_{u}^H\right]\\\notag &= E\left.\left[E\left.\left[\theta(H(t)-a_1
t - H(u)+a_1 u  ) + (1-\theta)(H(t)-a_0 t - H(u)+a_0 u )\right|
\theta,\F_{u}^H\right]\right| \F_{u}^H\right]\\\notag &\ \ +
E\left.\left[\int_u^{t} E\left.\left[(\theta -
\tilde{p}_{s-})\right| \F_{s-}^H\right](a_1-a_0)ds\right|
\F_{u}^H\right] =0.\\\notag
\end{align}
It follows that the process $\left(H(t)-\int_0^t (\tilde{p}_{s-}a_1
+ (1-\tilde{p}_{s-})a_0) ds\right)$ is a martingale with respect to
$\F_t^H$ under the probability measure $P_p$. Therefore, the
predictable process $\left(\int_0^t (\tilde{p}_{s-}a_1 +
(1-\tilde{p}_{s-})a_0) ds\right)$ is the (predictable) compensator
of the process $(H(t))$ with respect to $\F_t^H$ under the
probability measure $P_p$.
\end{proof}

Lemmas \ref{lem:time_change} and \ref{compensator} yield the
following corollary:

\begin{cor}\label{martingales}
$\\$
C1. The (predictable) compensator of the process $(X(\K(t)))$ is \\
$\left(\int_0^t (p_{s-}\mu_1 + (1-p_{s-})\mu_0) d\K(s)\right).$ \\
C2. The (predictable) compensator of the process is\\
$\left(X(\K(t)) -\int_{\mathbb{R}\setminus \{0\}} h
\widetilde{N}_{\nu_0}(\K(t),dh) - \mu_0\K(t)\right)$
is $\left(- \beta\sigma^2 \int_0^t p_{s-} dK(s)\right).$ \\
C3. The (predictable) compensator of the random measure $
N(d\K(t),dh)$ is $(p_{s-}  \nu_1(dh) + (1-p_{s-})  \nu_0(dh))
d\K(s)$; see Jacod and Shiryaev (1987, Ch.~II, Theorem 1.8)
\cite{Jacod}.
\end{cor}


The following lemma states the the posterior process $(p_t)$ spends
zero time at any given positive contour-line lower than $1$.
\begin{lem}\label{spends}
For every $t\geq 0$, $p\in [0,1]$, and every $0<\delta<1$,
$$E^{0,p}\left[\int_0^t \chi_{\{p_{s-}=\delta\}}ds\right]=0.$$
\end{lem}

\begin{proof}

\begin{align}\notag
&E^{0,p}\left[\int_0^t \chi_{\{p_{s-}=\delta\}}ds\right] =
E^{0,p}\left[\int_0^t \chi_{\{p_s=\delta\}}ds\right] \\\notag &=
E\left.\left[\int_0^t \chi_{\{p_s=\delta\}}ds\right| \theta=1\right]
p + E\left.\left[\int_0^t
\chi_{\{p_s=\delta\}}ds\right|\theta=0\right] (1-p) \\\notag
&=E\left.\left[\int_0^t
\chi_{\{\ln\left(\frac{1-p_s}{p_s}\right)=\ln\left(\frac{1-\delta}{\delta}\right)\}}ds\right|\theta=1\right]
p + E\left.\left[\int_0^t
\chi_{\{\ln\left(\frac{1-p_s}{p_s}\right)=\ln\left(\frac{1-\delta}{\delta}\right)\}}ds\right|\theta=0\right]
(1-p)=0.\\\notag
\end{align}
The last equation follows since, as long as jumps from $B_\infty$ do
not appear (see Remark \ref{key}), the process
$\left(\ln\left(\frac{1-p_s}{p_s}\right)\right)$ is a time change of
a L\'{e}vy process whose jump process part has finite variation and
has no positive jumps, given the type $\theta$ (See Bertoin (1996,
Ch.~V, Theorem 1) \cite{Bertoin}). In case a jump from $B_\infty$
appears, from that time onwards the posterior process $(p_t)$
remains at level $1$.
\end{proof}

The following lemma assures that $\alpha^*$ is well defined.
\begin{lem}\label{alpha}
Eq.~(\ref{eta}) admits a unique solution in the interval $(0,\infty)$.\\
\end{lem}

\begin{proof}[\textbf{Proof}]
$\newline$ The function $f$ is a continuous function that satisfies
$f(0)<0$ and $f(\infty) = \infty$. To show that $f(\alpha)=0$ has a
unique solution, it is therefore sufficient to prove that $f$ is
increasing in $\alpha$. Note that (if $\sigma\neq 0$ then)
$\frac{1}{2}(\alpha+1)\alpha
\left(\frac{\mu_1-\mu_0}{\sigma}\right)^2 - r $ is increasing in
$\alpha$. It remains to prove that if $\vg-\vvg>0$, i.e., $\vi
(\mathbb{R}\backslash\{0\})-\vvi(\mathbb{R}\backslash\{0\})>0$, then
$\int_{\mathbb{R}\setminus \{0\}}
\left(\left(\frac{\nu_0}{\nu_1}(h)\right)^{\alpha}-1\right) \vvh+
\alpha (\vg-\vvg)$ is increasing in $\alpha$. Since
\begin{align}\notag
& \int_{\mathbb{R}\setminus
\{0\}}\left[\left(\left(\frac{\nu_0}{\nu_1}(h)\right)^{\alpha}-1\right)\vvh
+ \alpha (\vh-\vvh) \right]\\\notag &=\int_{\mathbb{R}\setminus
\{0\}}\left[\left(\frac{\nu_0}{\nu_1}(h)\right)\left(\left(\frac{\nu_0}{\nu_1}(h)\right)^{\alpha}-1\right)
+ \alpha \left(1-\frac{\nu_0}{\nu_1}(h)\right) \right]\vh
\end{align}
and $$\underset{ \{ h\mid\frac{\nu_0}{\nu_1}(h)=1 \}
}{\int}\left[\left(\frac{\nu_0}{\nu_1}(h)\right)\left(\left(\frac{\nu_0}{\nu_1}(h)\right)^{\alpha}-1\right)
+ \alpha \left(1-\frac{\nu_0}{\nu_1}(h)\right) \right]\vh=0$$ it is
sufficient to prove that for $\vi$-a.e. $h\in \{ h\mid
\frac{\nu_0}{\nu_1}(h)\neq 1  \} $,
$$g_h(\alpha) = \left(\frac{\nu_0}{\nu_1}(h)\right)\left(\left(\frac{\nu_0}{\nu_1}(h)\right)^{\alpha}-1\right) + \alpha \left(1-\frac{\nu_0}{\nu_1}(h)\right)  $$ is increasing\footnote{In our model $\frac{\nu_0}{\nu_1}(h)\neq 1$ is actually $\frac{\nu_0}{\nu_1}(h)< 1$ $\vi$-a.s.
Yet, the proof works in the more general case of inequality.} in
$\alpha$. Now,
\begin{align}\notag
g_h'(\alpha) &=  \left(\frac{\nu_0}{\nu_1}(h)\right)^{\alpha+1}
\ln\left(\frac{\nu_0}{\nu_1}(h) \right)+
\left(1-\frac{\nu_0}{\nu_1}(h)\right)
\\\notag
&> \left(\frac{\nu_0}{\nu_1}(h)\right)
\ln\left(\frac{\nu_0}{\nu_1}(h) \right)+
\left(1-\frac{\nu_0}{\nu_1}(h)\right) > 0,
 \end{align}
where the first inequality holds since $\alpha>0$ and the second
inequality holds since $x\ln(x) +1-x > 0$ for every $x\neq 1$.
Therefore, $g_h(\alpha)$ is increasing, as desired.
\end{proof}

%
%
%

\textbf{Acknowledgments.} 
We thank Sven Rady for useful comments on an earlier version of the
paper. We also would like to thank an anonymous referee for
his\textbackslash her helpful comments and suggestions for
improvements to an earlier version of this paper. This research was
supported in part by Israel Science Foundation grant $\sharp
212/09$, and by the Google Inter-university Center for Electronic
Markets and Auctions.




\bibliographystyle{plain} 
\bibliography{bib_Asaf} 

\end{document}